\documentclass[12pt]{amsart}

\usepackage{amsfonts,amssymb}
\usepackage{amsthm}
\usepackage{verbatim}
\usepackage{color}
\usepackage{a4wide}
\theoremstyle{plain}
\usepackage{amsfonts,amssymb}
\usepackage{amsmath}

\newtheorem{theorem}{Theorem}[section]
\newtheorem{lemma}[theorem]{Lemma}
\newtheorem{proposition}[theorem]{Proposition}
\newtheorem{corollary}[theorem]{Corollary}
\newtheorem{Counter-example}[theorem]{Counter-example}
\newtheorem{remark}[theorem]{Remark}
\newtheorem{example}[theorem]{Example}

\theoremstyle{definition}

\theoremstyle{remark}

\long\def\symbolfootnote[#1]#2{\begingroup\def\thefootnote{\fnsymbol{footnote}}
\footnote[#1]{#2}\endgroup}

\begin{document}

\def\Q{\mathbb Q}
\def\R{\mathbb R}
\def\N{\mathbb N}
\def\Z{\mathbb Z}
\def\C{\mathbb C}
\def\S{\mathbb S}
\def\L{\mathbb L}
\def\H{\mathbb H}
\def\K{\mathbb K}
\def\X{\mathbb X}
\def\Y{\mathbb Y}
\def\Z{\mathbb Z}
\def\E{\mathbb E}
\def\J{\mathbb J}
\def\I{\mathbb I}
\def\T{\mathbb T}
\def\H{\mathbb H}

\title[Codimension two spacelike submanifolds]{
Spacelike Submanifolds of Codimension Two with Parallel Mean Curvature Vector Field in\\ Lorentz-Minkowski Spacetime Contained\\ in the Light Cone
}

\author{Francisco J. Palomo and Alfonso Romero}

\address{F.J. Palomo, Departamento de Matem\'{a}tica Aplicada, 
Universidad de M\'{a}laga,  29071-M\'{a}laga (Spain)}
\email{fpalomo@uma.es}

\address{A. Romero,
Departamento de Geometr\'{i}a y Topolog\'{i}a, Universidad de Granada, 18071-Granada (Spain)}
\email{aromero@ugr.es}

\date{}

\begin{abstract}
A general integral inequality is established for compact spacelike submanifolds of codimension two in the Lorentz-Minkowski spacetime under the assumption that the mean curvature vector field is parallel.  This inequality is then used to derive a rigidity result. Specifically,  
we obtain a complete characterization of all compact spacelike submanifolds with parallel mean curvature vector field that lie in the light cone of the Lorentz-Minkowski spacetime: they must be totally umbilical spheres contained in a spacelike hyperplane in Lorentz-Minkowski spacetime.  Our approach unifies and extends previous partial results, providing a dimension-independent technique that avoids additional restrictive assumptions.
Furthermore, when the spacelike submanifold with parallel mean curvature vector field is represented as a graph over the sphere $\S^n$ in the Lorentz-Minkowski spacetime $\mathbb{L}^{n+2}$ and lies in the light cone, the condition of parallel mean curvature vector field translates into an elliptic partial differential equation on $\S^n$, previously studied by M. Obata. As a consequence of our uniqueness result, we explicitly describe all solutions to Obata’s equation, thereby obtaining a new proof of his classical result.

\end{abstract}

\maketitle

\symbolfootnote[ 0 ]{Both authors were partially supported by the Spanish MICINN and ERDF project PID2024-156031NB-I00, and the second one was also partially supported by the ``Mar\'{\i}a de Maeztu'' Excellence Unit IMAG, reference CEX2020-001105-M, funded by MCIN-AEI-10.13039-501100011033.
}

\thispagestyle{empty}


\section{Introduction}
\noindent The notion of a parallel mean curvature vector field is the natural extension of the concept of constant mean curvature for hypersurfaces to the case of submanifolds with codimension $\geq 2$, \cite{Hoffman}, \cite{Chen}, \cite{Y}. The classification of complete submanifolds with parallel mean curvature vector field in space forms is a classical problem that has attracted considerable attention, as noted in \cite{Santos}. The standard approach, as described in \cite{Santos} and the references therein, relies on a Simons-type formula for the Laplacian of an appropriately chosen function with extrinsic geometric significance. This method has been generalized to include spacelike submanifolds in indefinite space forms, under the assumption that the normal bundle is negatively definite (see, for instance, \cite{Cheng}).

\vspace{1mm}

Submanifolds with parallel mean curvature vector field also admit a variational characterization. Specifically, a compact submanifold $M^n$ in Euclidean space $\mathbb{E}^m$ with $n+1\leq m$, has parallel mean curvature vector field if and only if it is a critical point of the $n$-dimensional area functional under volume-preserving variations. Similarly, a compact spacelike submanifold $M^n$ in Lorentz-Minkowski spacetime $\mathbb{L}^m$ with $n+2\leq m$ (since no compact spacelike hypersurface exists in $\L^m$), has parallel mean curvature vector field if and only if it is a critical point of the $n$-dimensional area functional under volume-preserving variations that also preserve the spacelike condition.

\vspace{1mm}

In the case of a spacelike immersion $\psi : M^n \rightarrow \mathbb{L}^{n+2}$ such that $\psi(M^n)$ is contained in the (future) light cone 
$\Lambda^{n+1}_+ = \{\,v\in \L^{n+2}\,:\,\langle v,v \rangle =0,\, v_{0}>0\,\},$
the immersion $\psi$ has parallel mean curvature vector field $\mathbf{H}$ if and only if $\|\mathbf{H}\|^2$ is constant (Proposition~\ref{equivalence2}).
We point out that the light cone $\Lambda^{n+1}_+$ is a degenerate hypersurface in $\L^{n+2}$; however, it plays a significant role in conformal Riemannian geometry. Specifically, a simply connected Riemannian manifold $M^n$ of dimension $n(\geq 3)$ is locally conformally flat if and only if it admits an isometric immersion in $\Lambda^{n+1}_+$ \cite{Bri} (revisited in \cite{AD} in a modern approach). This fact is one of the motivations to study spacelike submanifolds in $\Lambda^{n+1}_+ \subset \mathbb{L}^{n+2}$ (see, for instance, \cite{ACR}, \cite{IPR}).

\vspace{1mm}

In general, for a spacelike immersion $\psi: M^n \rightarrow \mathbb{L}^{n+2}$ such that $\psi(M^n)\subset \Lambda_+^{n+1}$, if $M^n$ is assumed to be compact, it must be diffeomorphic to the $n$-dimensional sphere $\mathbb{S}^n$ and $\psi$ is an embedding \cite[Prop. 5.1]{PPR}. Once $M^n$ is identified with $\mathbb{S}^n$, up to a diffeomorphism, see Proposition \ref{010325A}, the spacelike embedding $\psi : M^n \rightarrow \Lambda_+^{n+1} \subset \mathbb{L}^{n+2}$ is completely described by a smooth function $f : \mathbb{S}^n \rightarrow \mathbb{R}$ as $$\psi(x)=\psi_f(x)=e^{f(x)}(1,x),$$for all $x\in \mathbb{S}^n$. We will call $\Sigma^n_f:=\psi_f(\mathbb{S}^n)$ the spacelike graph over $\mathbb{S}^n$ in $\Lambda_+^n$ defined by $f$ in Section 4. The metric $g_f$ on $\mathbb{S}^n$ induced from $\mathbb{L}^{n+2}$ via $\psi_f$ satisfies $$g_f=e^{2f}g_0,$$ where $g_0$ is the round metric of constant sectional curvature $1$, (see Section 4). Therefore, both $n$-dimensional area functional and the volume constraint are naturally expressed in terms of the function $f$. Moreover, by means of Proposition \ref{equivalence2} and Remark \ref{110725A}, the following assertions for a spacelike immersion $\psi_f$ are equivalent:

\begin{quote}
    
{\it 
\begin{enumerate}
\item The mean curvature vector field $\mathbf{H}_f$ of $\psi_f$ is parallel.
\vspace*{1mm}
\item  $\|\mathbf{H}_f\|^2$ is a positive constant $k$.
\vspace*{1mm}
\item The scalar curvature of the induced metric $g_f$ is the constant $n(n-1)k$.

\item The function $f$ satisfies the following elliptic PDE 
\[
(\mathrm{E}) \hspace*{10mm} 2\Delta^{0}f+(n-2)\,\|\nabla^{0}f\|^{2}_{0}=n\,(1- k\,e^{2f}),\hspace*{20mm} 
\]
\noindent where $\Delta^{0}$, $\nabla^{0}$ and $\| \; \|_{0}^2$ denote the Laplacian, the gradient and the square of the norm for the metric $g_0$, respectively.
\end{enumerate}
}
\end{quote}

Equation (E) can thus be interpreted in two distinct ways. On one hand, it arises as the Euler–Lagrange equation of an extrinsic variational problem (see Section 4). On the other hand, it also has an intrinsic geometric interpretation: each solution $f$ of equation (E) determines a Riemannian metric, pointwise conformal to the standard round metric $g_0$ on $S^n$, with constant scalar curvature.
At this point, we recall the following classical result by Obata, stated in \cite[Prop. 6.1]{Ob71}.

\vspace{1mm}

\begin{quote}{\it 
Any pointwise conformal change of the usual round metric on a sphere $\S^n$ with constant scalar curvature is of constant sectional curvature.}
\end{quote}

\vspace{1mm}

\noindent We provide an alternative proof of this fact as a consequence of Theorem \ref{1106A}, 
Specifically, we prove the following uniqueness result, Corollary \ref{ultimo},

\vspace{1mm}

\begin{quote}{\it 
For each positive constant $k$, the only solutions of equation {\rm (E)} are the functions $f: \S^n \rightarrow \R$, $n \geq 2$, given by
\[
f(x)=\log\Big(\,\frac{1/\sqrt{k}}{-u_0+\langle u,x\rangle}\,\Big), \quad x\in \S^n,
\]
where $\langle u,x\rangle=\sum_{i=1}^{n+1}u_{i}x_{i}$, $u_0<0$ and $-u_0^2+\langle u,u\rangle=-1$. 
}
\end{quote}

\vspace{2mm}

\noindent In fact, this result not only allows us to reprove Obata’s result \cite[Prop. 6.1]{Ob71}, as stated in Corollary \ref{Obata}, but also explicitly provides all the conformal changes of metric involved.
Moreover, several remarks comparing Obata’s original proof with ours are given in Remark \ref{250225A}.
 
\vspace{1mm}

Now, we want to make explicit our main geometric result, Theorem \ref{1106A}, and the technique used in this paper. First of all, some partial results were already known in the study of spacelike submanifolds with parallel mean curvature vector field in $\mathbb{L}^n$.
For instance, if a spacelike surface $M^2$ with parallel mean curvature vector field in $\mathbb{L}^4$ has the topology of a 2-sphere, then it must be a totally umbilical sphere contained in a spacelike hyperplane of $\mathbb{L}^4$ \cite[Cor. 4.5]{AER}. This was derived from a general integral formula \cite[Th. 3.1]{AER}, which applies under the assumption that the compact spacelike submanifold is pseudo-umbilical.
In the 2-dimensional case, a Hopf-type quadratic differential can be defined on $M^2$, which vanishes identically due to the topological assumption \cite[Lemma 3.3]{AER}.
Although this implies that any compact spacelike surface $M^2$ with parallel mean curvature vector field in $\mathbb{L}^4$ lying in the light cone $\Lambda_{+}^3$ must be totally umbilical, a direct proof of this fact was later provided in \cite[Th. 5.4]{PaRo}, using a technique tailored to this specific setting, particularly, the Gauss-Bonnet theorem for $\mathbb{S}^2$.
Subsequently, it was shown in \cite[Cor. 5.4]{PPR} that a compact spacelike submanifold $M^4$ in $\mathbb{L}^6$ lying in the light cone $\Lambda_+^5$, with parallel mean curvature vector field and satisfying a certain lower bound on the scalar curvature, must also be totally umbilical. The proof again relied on dimension-specific methods, notably the Gauss-Bonnet-Chern-Avez theorem in dimension $4$, \cite{Avez}.
These examples illustrate that previous techniques depend heavily on the dimension and specific geometric assumptions. Therefore, a more general approach is needed (see also \cite{Chen1}).
Thus, an integral formula was developed in \cite{AER} that is useful for compact pseudo-umbilical spacelike submanifolds $M^n$ of $\mathbb{L}^{n+2}$. However, as shown in Proposition \ref{equivalence1}, when $\psi(M^n) \subset \Lambda_+^{n+1}$, the immersion $\psi$ is pseudo-umbilical if and only if it is totally umbilical.
Thus, a new technique that is particularly well-suited to our setting, i.e., independent of both the dimension of the spacelike submanifold and additional assumptions, is required. This has now been achieved in this paper: the integral formula obtained in Proposition \ref{2605C}. As a consequence, we prove a general integral inequality in Proposition \ref{040125B} that characterizes geometrically the equality case:
\begin{quote}{\it 
     Let $\psi: M^{n}\rightarrow \L^{n+2}$ be a compact spacelike submanifold with $\nabla^{\perp} \mathbf{H}=0$.
  Then, for every $a\in \L^{n+2}$, timelike with $\langle a,\xi_{0}\rangle >0$, the following inequality holds
$$
    \int_{M^n}\langle a, \xi_{0}-\xi_{1}\rangle \Big(n(n-1)\langle \mathbf{H},\mathbf{H} \rangle- S\Big)\,dV_g \, \geq 0\,,
$$
where $\xi_{0}, \xi_{1}\in \mathfrak{X}^{\perp}(M^n)$ are taken as in Corollary {\rm \ref{coro1}}.
Moreover, the equality holds if and only if  $M^n$ is a totally umbilical sphere with constant sectional curvature $\langle \mathbf{H},\mathbf{H} \rangle >0$, in a spacelike hyperplane of $\L^{n+2}$. }
\end{quote}
\vspace{1mm}
\noindent After several consequences of this result, the main result, Theorem \ref{1106A}
solves the geometric version of the previously quoted variational problem:

\vspace{1mm}

\begin{quote}{\it 
    Every compact spacelike submanifold $\psi: M^{n}\rightarrow \Lambda_+^{n+1}\subset \L^{n+2}$, with parallel mean curvature vector field, is a totally umbilical round sphere $\S^n(v,r):=\{x\in \L^{n+2}\, :\, \langle x,x\rangle = 0, \langle v,x\rangle = r\}$, where $v=(v_0,...,v_{n+1})\in \L^{n+2}$ with $\langle v,v\rangle = -1$, $v_0<0$ and $1/r^2=\langle \mathbf{H},\mathbf{H} \rangle$, $r>0$, is its $($constant$)$ sectional curvature. }
\end{quote}

\vspace{1mm}

\hyphenation{Lo-rent-zi-an}

Finally, let us outline the remainder of the content of the paper. Section 2 reviews basic concepts and establishes the notation and conventions used throughout the article. At the same time, we highlight several notable facts concerning spacelike submanifolds of codimension two with parallel mean curvature vector field in the Lorentz-Minkowski spacetime. Particularly relevant for our purposes is Corollary \ref{coro1}, which provides two lightlike normal vector fields adapted to the study of compact spacelike submanifolds with parallel mean curvature vector field.
Section 3 recalls key properties of codimension two spacelike submanifolds lying in the light cone, with the main references being \cite{PaRo} and \cite{PPR}. In particular, Propositions \ref{equivalence1} and \ref{equivalence2} characterize, respectively, the conditions for total umbilicity and for having parallel mean curvature vector field in the context of spacelike submanifolds through the light cone.
Remark \ref{examples} provides an explicit description of totally umbilical compact $n$-dimensional spacelike submanifolds lying in the light cone as the $n$-spheres previously quoted $\S^n(v,r)$.

\vspace{1mm}

Section 4 focuses on the study of spacelike graphs $\Sigma^n_f$ over $\mathbb{S}^n$ lying in the light cone. In this case, the squared length of the corresponding mean curvature vector field depends only of the function $f$, and it is explicitly given in formula (\ref{0403A}). This leads to the key equation (E) for spacelike graphs with parallel mean curvature vector field, as well as to a variational interpretation of the function $f$ involved.
\noindent It should be remarked that, according to Proposition \ref{equivalence2}, the assumption on the mean curvature vector field can be replaced by any of the following equivalent conditions:
\begin{enumerate}{\it 
    \item The mean curvature vector field $\mathbf{H}_f$ satisfies $\langle \mathbf{H}_f,\mathbf{H}_f \rangle =$constant, 
\item The scalar curvature $S_f$ of the induced metric is constant.
\item The squared length $\langle\Pi_f, \Pi_f \rangle$ of the second fundamental form, $\Pi_f$, is constant.
\item The normalized mean curvature vector field $\big(1/\sqrt{|\langle \mathbf{H}_f,\mathbf{H}_f \rangle|}\, \big)\,\mathbf{H}_f$ is parallel.
}
\end{enumerate}

\vspace{1mm}

To end this Introduction, note that Theorem \ref{1106A} does not hold if the compactness assumption is weakened to completeness, as Counter-example \ref{counter-example}\textbf{(a)} shows. Moreover, as a direct application of Theorem \ref{1106A}, Corollary \ref{equivalence3} improves Proposition \ref{equivalence2} by adding the equivalence between the parallelism of the mean curvature vector field and the parallelism of the second fundamental form for compact spacelike submanifolds $M^n$ in $\mathbb{L}^{n+2}$ through the light cone. However, Counterexample \ref{counter-example}\textbf{(b)} provides an example showing that this equivalence is not true if the spacelike submanifold is assumed to be complete and non-compact.

\section{Preliminaries}
 
\noindent Let $\L^{n+2}$ be the $(n+2)$-dimensional Lorentz-Minkowski spacetime, that is, $\L^{n+2}$ is $\R^{n+2}$ endowed
with the Lorentzian metric 
$
\langle\, ,\, \rangle =-dx_{0}^2+
\sum\limits_{\scriptscriptstyle i=1}^{\scriptscriptstyle n+1}dx_{i}^{2}
$ in canonical coordinates. An immersion $\psi:M^n\rightarrow \L^{n+2}$ of an $n(\geq 2)$ dimensional (connected and orientable) manifold $M^n$ in $\L^{n+2}$ is said to be
spacelike if the induced metric $M^n$ by $\psi$ (denoted also by the same symbol
$\langle\; ,\; \rangle$) is Riemannian. In this
case, $M^n$ is also called a spacelike submanifold of $\L^{n+2}$. This paper will be interested in the case $M^n$ is compact. Therefore, the codimension must be at least $2$ (see \cite{Harris} for instance).

\vspace{1mm}

Let $\overline{\nabla}$ and $\nabla$ be the Levi-Civita
connections of $\L^{n+2}$ and $M^n$, respectively, and let
$\nabla^{\perp}$ be the normal bundle connection. The basic relations  between these connections are the well-known Gauss and Weingarten formulas 
\begin{equation}\label{Gauss_Weingarten}
\overline{\nabla}_X Y=\nabla_XY+ \mathrm{II}(X,Y)
\, \quad \mathrm{and} \, \quad
\overline{\nabla}_X\xi=-A_{\xi}X+\nabla^{\perp}_X\,\xi,
\end{equation}
for any $X,Y \in \mathfrak{X}(M^{n})$ and $\xi \in
\mathfrak{X}^{\perp}(M^{n})$. Here $\mathrm{II}$ denotes the
second fundamental form of the spacelike submanifold and $A_{\xi}$ the  Weingarten operator
corresponding to $\xi$. 
The shape operator $A_{\xi}$ is related to $\mathrm{II}$ by
\begin{equation}\label{relacion}
\langle A_{\xi}X, Y \rangle = \langle \mathrm{II}(X,Y), \xi
\rangle,
\end{equation}
for any $X,Y \in \mathfrak{X}(M^{n})$. 

\vspace{1mm}
 
 Since we assume $M^n$ to be orientable, we may take two globally defined, everywhere independent lightlike normal vector fields $\xi, \eta\in \mathfrak{X}^{\perp}(M^n)$ with $\langle \xi,\eta\rangle=1$. This assertion follows from the following argument: 

\vspace{1mm}

Any unit timelike vector $a\in \L^{n+2}$ can be decomposed into tangential and normal components 
\begin{equation}\label{decomposition_1}
a=a_x^{\top}+a_x^{\mathrm{N}},\quad \text{at}\; \text{each}\quad x\in M^n,    
\end{equation}
according to the orthogonal decomposition 
\begin{equation}\label{descomposicion_2}
\L^{n+2}= T_{x}M^n\oplus T_{x}^{\perp}M^n.
\end{equation}
Thus, we obtain $a^{\top}\in \mathfrak{X}(M^n)$ and 
$a^{\mathrm{N}} \in \mathfrak{X}^{\perp}(M^{n})$ that satisfy

\begin{equation}\label{modulos}
-1=\langle a,  a\rangle=\langle a^{\top},  a^{\top}\rangle+ \langle a^{\mathrm{N}},  a^{\mathrm{N}}\rangle,
\end{equation}
and hence  $\langle a^{N},  a^{N}\rangle=-1-\langle a^{\top},  a^{\top}\rangle\leq -1$, because $\langle a^{\top},  a^{\top}\rangle \geq 0$. Thus, 
\begin{equation}\label{normal_temporal_unitario}
a^{\psi}:=\frac{1}{{\sqrt{1+\|a^{\top}\|^{2}}}}\,a^{\mathrm{N}}
\end{equation}
is a unit timelike normal vector field globally defined on $M^n$.

\vspace{1mm}

Now, fix an orientation on $M^n$ and consider the usual positive orientation on $\L^{n+2}$. An orthonormal basis $(\alpha_{1}, \alpha_{2})$ of the Lorentzian plane $T^{\perp}_{x}M^n$ is said to be positively oriented if for a positively oriented orthonormal basis $(e_1,\dots, e_n)$ in $T_{x}M^n$ (and hence for any positively orthonormal basis) the orthonormal basis $(\alpha_1,\alpha_2,e_1,\dots, e_n)$ in $\L^{n+2}$ is positively oriented. Now consider $\alpha_1=a_x^{\psi}$ as above and define $Ja_x^{\psi}\in T^{\perp}_{x}M^n$ 
as the unique normal vector at this point that satisfies $\langle Ja_x^{\psi},  Ja_x^{\psi}\rangle=1$,  $\langle Ja_x^{\psi},  a_x^{\psi}\rangle=0$ with the orthonormal basis $(a_x^{\psi}, Ja_x^{\psi})$ in  $T_{x}^{\perp}M$ positively oriented. In this way, we have a globally defined unit spacelike normal vector field $Ja^{\psi}$. Finally, we can take 
\begin{equation}\label{291024A}
    \xi=\frac{1}{\sqrt{2}}(a^{\psi}+Ja^{\psi})\; \textrm{ and }\; \eta=\frac{1}{\sqrt{2}}(-a^{\psi}+Ja^{\psi}).
\end{equation}

\begin{remark}
{\rm 
If we change $\L^{n+2}$ to any $(n+2)$-dimensional orientable spacetime $\overline{M}^{n+2}$, then the same conclusion as above is obtained starting from a timelike vector field on $\overline{M}^{n+2}$. This technical result does not hold in general for codimension $\geq 3$.
}
\end{remark}

Let $\psi : M^n \rightarrow \L^{n+2}$ be a spacelike submanifold. Let us consider any pair $\xi, \eta \in \mathfrak{X}^{\perp}(M^n)$ such that $\langle \xi, \xi\rangle=\langle \eta, \eta\rangle =0$ and $\langle \xi, \eta\rangle =1$, everywhere on $M^n$. We can write 
the following (global) formula for the second fundamental form
\begin{equation}\label{segunda_forma_fundamental}
\mathrm{II}(X,Y)=\langle A_{\eta}X, Y \rangle\,\xi+\langle A_{\xi}X, Y \rangle\,\eta,
\end{equation}
for every $X,Y\in \mathfrak{X}(M^{n})$. 

\vspace{1mm}

The mean curvature vector field of the submanifold is defined by
$\mathbf{H}=(1/n)\,\mathrm{trace}_{_{\langle\; , \;
\rangle}}\mathrm{II}.$ Using (\ref{segunda_forma_fundamental}), we can express
\begin{equation}\label{curme}
\mathbf{H}=\frac{1}{n}\big(\,(\mathrm{trace}\,
A_{\eta})\,\xi+(\mathrm{trace}\, A_{\xi})\,\eta\,\big),
\end{equation}
and consequently 
\begin{equation}\label{squared_lenght_H}
\langle \mathbf{H},\mathbf{H} \rangle = \frac{2}{n^2}\,(\mathrm{trace}A_{\xi})(\mathrm{trace}A_{\eta}).
\end{equation} 

\begin{remark}\label{1908B} 
{\rm It should be pointed out that $\mathbf{H}$ does not have the same causal character everywhere on $M^n$, in general. However, if $M^n$ is assumed to be compact (and hence $\mathbf{H}\neq 0$), then it is impossible that $\mathbf{H}$ is timelike everywhere \cite[Rem. 4.2]{AER} or lightlike everywhere \cite{MS03}.   
}
\end{remark}

On the other hand, the Gauss equation $\psi$ is written 
 \begin{equation}\label{1}
R(X,Y)Z= A_{_{\mathrm{II}(Y,Z)}}X-A_{_{\mathrm{II}(X,Z)}}Y,
\end{equation}
where $R$ is the curvature tensor of $M^n$. By contraction, we obtain
\begin{equation}\label{3}
\mathrm{Ric}(Y,Z)=n\,\langle A_{\mathbf{H}}Y,Z\rangle-\big\langle
\,(A_{\xi}A_{\eta}+ A_{\eta}A_{\xi}\,)Y, Z \big\rangle,
\end{equation}
where Ric is the Ricci tensor of $M^n$.
Therefore, using (\ref{curme}) we get from (\ref{3})
\begin{equation}\label{101}
S=2\,(\mathrm{trace}A_{\xi})(\mathrm{trace}A_{\eta})-2\,\mathrm{trace}(A_{\xi}A_{\eta}),
\end{equation}
where $S$ is the scalar curvature of $M^{n}$.

\vspace{1mm}

Furthermore, the squared length of the second fundamental form is given by
$\langle
\mathrm{II},\mathrm{II}\rangle_{x}=\sum_{i,j=1}^{n}\langle
\mathrm{II}(e_{i},e_{j}),\mathrm{II}(e_{i},e_{j})\rangle,$
where $(e_{1},\dots,e_{n})$ is an orthonormal basis of $T_{x}M^{n}$, $x\in M^n$. Therefore, from (\ref{segunda_forma_fundamental}) we get 
\begin{equation}\label{squared_lenght_second_fundamental_form}
\langle
\mathrm{II},\mathrm{II}\rangle = 2\,\mathrm{trace}(A_{\xi}A_{\eta}).
\end{equation}
Thus, formula (\ref{101}) may be also written as follows    
\begin{equation}\label{101a}
S=n^2\,\langle
\mathbf{H},\mathbf{H} \rangle-\langle \mathrm{II},\mathrm{II}\rangle\,.
\end{equation}

Finally, we have the Codazzi equation

\begin{equation}\label{Codazzi_1}
(\widetilde{\nabla}_X\mathrm{II})(Y,Z)=(\widetilde{\nabla}_Y\mathrm{II})(X,Z),
\end{equation}
where
$$(\widetilde{\nabla}_X\mathrm{II})(Y,Z)=\nabla^{\perp}_X\,\mathrm{II}(Y,Z)-\mathrm{II}(\nabla_XY,Z)-\mathrm{II}(Y,\nabla_XZ),$$
for any $X,Y,Z \in \mathfrak{X}(M^n)$. The Codazzi equation is now equivalently written as
\begin{equation}\label{Codazzi_2}
(\nabla_{X}A_{\mu})Y-(\nabla_{Y}A_{\mu})X=
A_{\nabla^{\perp}_{X}\mu}Y - A_{\nabla^{\perp}_{Y}\mu}X,
\end{equation}
for any $\mu\in\mathfrak{X}^{\perp}(M^{n})$.

\vspace{1mm}

Taking into account the codimension two of the immersion $\psi$ and that $\nabla^{\perp}$ is metric, we have a well-defined one-form $\alpha$ on $M^n$ given by the equivalent conditions
\begin{equation}\label{one-form}
\nabla^{\perp}_{X}\xi = \alpha(X)\xi, \qquad \nabla^{\perp}_{X}\eta= -\alpha(X)\eta,
\end{equation}
for any $X\in \mathfrak{X}(M)$.
Explicitly we have $\alpha(X)=\langle\nabla^{\perp}_{X}\xi ,\eta \rangle$. 
Hence,
\begin{equation}\label{normal_curvature}
d\alpha(X,Y)=\langle R^{\perp}(X,Y)\xi,\eta\rangle\,,
\end{equation}
for any $X,Y\in\mathfrak{X}(M^n)$, where $R^{\perp}$ is the normal curvature tensor of the spacelike submanifold. 
\begin{remark}\label{cambios} {\rm Obviously, the one-form $\alpha$ depends on the lightlike normal reference $(\xi,\eta)$ chosen. If for $f\in C^{\infty}(M^n)$, never vanishing, we consider $(\bar{\xi},\bar{\eta})$ defined by $\bar{\xi}=(1/f)\xi$, $\bar{\eta}=f\eta$, then the corresponding one-forms are related by $\bar{\alpha}=\alpha - d \log |f|$, hence $d\bar{\alpha}=d\alpha$, which can be also deduced from (\ref{normal_curvature}).
}
\end{remark}

\vspace{1mm}

Note that (\ref{Codazzi_2}) gives, taking into account (\ref{one-form}), the following equation 
\begin{equation}\label{Codazzi_3}
(\nabla_{X}A_{\eta})Y-(\nabla_{Y}A_{\eta}) X= -\alpha(X)
A_{\eta}Y + \alpha(Y) A_{\eta}X.
\end{equation}
On the other hand, from (\ref{curme}) we have that $\nabla^{\perp}\mathbf{H}=0$ is equivalent to
\begin{equation}\label{1006A}
\nabla  \mathrm{trace}(A_{\xi})=\mathrm{trace}(A_{\xi})\,\alpha^{\sharp} \quad \text{ and }\quad \nabla \mathrm{trace}(A_{\eta})=-\mathrm{trace}(A_{\eta})\, \alpha^{\sharp},
\end{equation}
where $\alpha^{\sharp}\in \mathfrak{X}(M^n)$ is given by $\langle \alpha^{\sharp}, X\rangle= \alpha(X)$ for all $X\in  \mathfrak{X}(M^n).$
Now, as a direct consequence of formulas (\ref{1006A}), we obtain
\begin{lemma}\label{lema1}
Let $\psi : M^n \rightarrow \L^{n+2}$ be a spacelike submanifold with $\nabla^{\perp}\mathbf{H}=0$, then the functions $\mathrm{trace}(A_{\xi})$ and $\mathrm{trace}(A_{\eta})$ satisfy
$$
\Delta \mathrm{trace}(A_{\xi})=(\| \alpha^{\sharp}\|^2+ \mathrm{div}\,\alpha^{\sharp})\,\mathrm{trace}(A_{\xi}),
$$
$$
\Delta \mathrm{trace}(A_{\eta})=(\| \alpha^{\sharp}\|^2- \mathrm{div}\,\alpha^{\sharp})\,\mathrm{trace}(A_{\eta}).
$$
\end{lemma}

In general, observe that if $\nabla^{\perp}\xi=\nabla^{\perp}\eta=0$ and the functions $\mathrm{trace}(A_{\xi})$ and  $\mathrm{trace}(A_{\eta})$ are constant, then we obtain $\nabla^{\perp}\mathbf{H}=0$ from (\ref{curme}). In the compact case and after a suitable rescaling, the converse is also true, as shown in the following result

\begin{corollary}\label{coro1}
Let $\psi : M^n \rightarrow \L^{n+2}$ be a spacelike submanifold such that  $\nabla^{\perp}\mathbf{H}=0$. If $M^n$ is compact, then there are two globally defined lightlike normal vector fields $\bar{\xi}, \bar{\eta}\in
\mathfrak{X}^{\perp}(M^{n})$ with $\langle \bar{\xi},\bar{\eta}\rangle=1$ such that $\nabla^{\perp}\bar{\xi}=\nabla^{\perp}\bar{\eta}=0$ $($equivalently, the corresponding one-form satisfies $\bar{\alpha}=0$$)$ and the functions $\mathrm{trace}(A_{\bar{\xi}})$ and  $\mathrm{trace}(A_{\bar{\eta}})$ are non-zero constants.
\end{corollary}
\begin{proof}
Let us choose any pair $\xi, \eta \in \mathfrak{X}(M^n)$ satisfying $\langle \xi, \xi\rangle=\langle \eta, \eta\rangle =0$ and $\langle \xi, \eta\rangle =1$ everywhere on $M^n$. Put $u=\mathrm{trace}(A_{\xi})$ and $v=\mathrm{trace}(A_{\eta})$. Using (\ref{squared_lenght_H}), the assumption $\nabla^{\perp}\mathbf{H}=0$ implies that the function $uv$ is constant, indeed a positive constant (Remark \ref{1908B}). Now, taking $\bar{\xi}=(1/u)\,\xi$ and $\bar{\eta}=u\,\eta$ we get $\mathrm{trace}(A_{\bar{\xi}})=1$ and $\mathrm{trace}(A_{\bar{\eta}})=uv$. Therefore, (\ref{1006A}) ends the proof.
\end{proof}

\begin{remark}
    {\rm Taking into account that $\bar{\xi}$ is parallel for the normal connection, a direct computation shows that for every $X\in \mathfrak{X}(M)$, we have
    $$
    \nabla^{\perp}_{X}\xi=\frac{1}{u}du(X)\,\xi.
    $$
    On the other hand, note that the compactness assumption on $M^n$ is only needed to assert that the constant $\langle \mathbf{H},\mathbf{H}\rangle$ is positive.}
\end{remark}

\section{Spacelike submanifolds in the light cone}

\noindent Let $\Lambda^{n+1}_+ = \{\,v\in \L^{n+2}\,:\,\langle v,v \rangle
=0,\, v_{0}>0\,\}$ be the future light cone of $\L^{n+2}$ (at the origin).  A spacelike submanifold $\psi:M^{n}\rightarrow \L^{n+2}$ factors through the light cone if $\psi(M^{n})\subset \Lambda^{n+1}_+$. An extrinsic characterization of such spacelike submanifolds is as follows: The necessary and sufficient condition for a spacelike submanifold $\psi: M^{n}\rightarrow \L^{n+2}$ factors through $\Lambda^{n+1}_+$ (after possibly performing a translation or the reflection respect to the spacelike hyperplane $x_0=0$) is the existence of a lightlike normal vector field $\xi$, that satisfies $\nabla^{\perp}\xi=0$ and $A_{\xi}=-\mathrm{Id}$
 \cite[0.9]{Mag} (compare with \cite[Th. 4.3]{IPR}). 

\vspace{1mm}
 
Now, we collect several results, proven in \cite{PPR} and \cite{PaRo}, on a spacelike submanifold $\psi: M^{n}\rightarrow \L^{n+2}$ through the light cone $\Lambda^{n+1}_+$.  Let us write $\psi=(\psi_{0},\psi_1, \cdots , \psi_{n+1})$. The normal vector fields
\begin{equation}\label{normal_vector_fields}
\xi=\psi\,\,\quad \textrm{and}\,\,\quad \eta=\frac{1+\| \nabla
\psi_{0}\|^2 }{2\,\psi_{0}^2}\, \xi -
\frac{1}{\psi_{0}}\,\Big(\partial_{0}\circ \psi +
\psi_{*}(\nabla \psi_{0})\Big),
\end{equation} 
where $\partial_{0}=\partial/\partial x_0$, are lightlike, and satisfy $\langle \xi, \eta \rangle =
1$ with $\nabla^{\perp}\xi = \nabla^{\perp}\eta =0$. 

\begin{remark}
    {\rm  For a spacelike submanifold $\psi: M^{n}\rightarrow \L^{n+2}$ through $\Lambda^{n+1}_+$ and the unit timelike vector $a=(1, 0,\cdots , 0)\in \L^{n+2}$, we have 
    $a^{\mathrm{N}}=\partial_{0}\circ \psi +
\psi_{*}(\nabla \psi_{0})$ and therefore from (\ref{normal_vector_fields}), we get
    $$
    a^{\psi}=\frac{1}{\sqrt{1+\|\nabla \psi_{0}\|^2}}\Big(\partial_{0}\circ \psi +
\psi_{*}(\nabla \psi_{0})\Big)=\frac{\sqrt{1+\|\nabla \psi_{0}\|^2}}{2\psi_{0}}\,\xi- \frac{\psi_{0}}{\sqrt{1+\|\nabla \psi_{0}\|^2}}\,\eta.
    $$
Now, a direct computation shows that, up to sign,
$$
J a^{\psi}=\frac{\psi_{0}}{\sqrt{1+\|\nabla \psi_{0}\|^2}}\,\Big(\frac{1+\|\nabla \psi_{0}\|^2}{2\psi_{0}^2}\,\xi +\eta \, \Big),
$$
and then
$$
 a^{\psi}+J a^{\psi}=\frac{\sqrt{1+\|\nabla \psi_{0}\|^2}}{\psi_{0}}\xi \quad \textrm{ and } \quad  -a^{\psi}+J a^{\psi}= \frac{2\psi_{0}}{\sqrt{1+\|\nabla \psi_{0}\|^2}}\eta.
$$
Hence, the lightlike normal vector fields corresponding to $a=(1, 0,\cdots , 0)\in \L^{n+2}$ and given in (\ref{291024A})  essentially agree with the vector fields given in (\ref{normal_vector_fields}).

 }
\end{remark}

Recall that we also have
\begin{equation}\label{Weingarten_operators}
A_{\xi}=-\mathrm{Id} \,\,\quad \mathrm{and}\,\,\quad
A_{\eta}= -\frac{1+\| \nabla \psi_{0}\|^2
}{2\,\psi_{0}^{2}}\mathrm{Id}+\frac{1}{\psi_{0}}\,\nabla^{2}
\psi_{0}, 
\end{equation}
where $\nabla^{2} \psi_{0} (v)=\nabla_{v}(\nabla
\psi_{0})$, for every $v\in T_{x}M^{n}$ and $x\in M^{n}$. Hence, we obtain 
\begin{equation}\label{traces}
\text{trace}(A_{\xi})=-n \,\,\quad \mathrm{and}\,\,\quad \text{trace}(A_{\eta}) = -n\frac{1+\| \nabla \psi_{0}\|^2}{2\,\psi_{0}^{2}}+\frac{\Delta \psi_{0}}{\psi_{0}}\,.
\end{equation}
Taking into account (\ref{squared_lenght_second_fundamental_form}) and previous formulas, equation (\ref{101a}) reduces now to
\begin{equation}\label{2805A}
S=-2(n-1)\text{trace}(A_{\eta})=n(n-1)\langle \mathbf{H},\mathbf{H} \rangle .
\end{equation}
Moreover, using the formula above, equation (\ref{101a}) also yields  
\begin{equation}\label{260725A}
 n \langle \mathbf{H},\mathbf{H} \rangle  =\langle \Pi, \Pi \rangle . 
\end{equation}

\begin{remark}\label{n=2}{\rm 
In the particular case $n=2$, formula (\ref{2805A}) becomes $K=\langle \mathbf{H},\mathbf{H} \rangle$, where $K$ denotes the Gauss curvature of $M^2$.
This identity was shown to be an important tool in the proof of \cite[Th. 5.4]{PaRo}.}
\end{remark}
Therefore, formula (\ref{curme}) is given in this case by
\begin{equation}\label{mean_curvature_2}
\mathbf{H}=-\frac{S}{2n(n-1)}\,\xi-\eta\,.
\end{equation}

\begin{remark}\label{110725A}
    {\rm Taking into account (\ref{traces}), formula (\ref{2805A}) is equivalent to
\begin{equation}\label{110725B}
\langle\mathbf{H}, \mathbf{H}\rangle=\frac{1+\| \nabla \psi_{0}\|^2
}{\psi_{0}^{2}}-\frac{2}{n\psi_{0}}\,\Delta 
\psi_{0}.
\end{equation}
    For $M^n$ compact and taking $x^{0}\in M^n$, where $\psi_0$ reaches its maximum value, we arrive to 
    $\langle\mathbf{H}, \mathbf{H}\rangle (x^0)\geq 1/\psi_{0}(x^0)>0.$
   Thus, if we assume $\langle\mathbf{H}, \mathbf{H}\rangle=k$ constant, then $k>0$.
   } 
\end{remark}

Observe that we have $A_{\mathbf{H}}=(S/n(n-1))\,\mathrm{Id}-A_{\eta}$, from (\ref{mean_curvature_2}) and (\ref{Weingarten_operators}). Consequently, we can state the following result, which shows how the extrinsic geometry of $\psi$ is encoded in the lightlike normal vector field $\eta$.

\begin{proposition}\label{equivalence1}
For any spacelike submanifold $\psi : M^n \rightarrow \L ^{n+2}$ such that 
$\psi(M^n)\subset \Lambda^{n+1}_+$, the following assertions are equivalent: 

\begin{enumerate}

\item[(i)] $\psi$ is totally umbilical, 

\item[(ii)] The normal vector field $\eta$ is umbilical, 

\item[(iii)] $\psi$ is pseudo-umbilical $($i.e., the mean curvature vector field $\mathbf{H}$ is umbilical$)$.
\end{enumerate}
\end{proposition}

On the other hand, the parallelism of the mean curvature vector field can be equivalently characterized as follows:

\begin{proposition}\label{equivalence2}
For any spacelike submanifold  $\psi:M^{n}\rightarrow \L^{m+2}$ such that 
$\psi(M^n)\subset \Lambda^{n+1}_+$, the following assertions are equivalent:
\begin{enumerate}
\item  The mean curvature vector field $\mathbf{H}$ of $\psi$ is  parallel,
\item The mean curvature vector field satisfies $\langle \mathbf{H},\mathbf{H} \rangle =$constant, 
\item The scalar curvature $S$ of the induced metric on $M^n$ is constant.
\item The squared length $\langle\Pi, \Pi \rangle$ of the second fundamental form, $\Pi$, is constant.
\end{enumerate}
If in addition, we have $\langle \mathbf{H}, \mathbf{H} \rangle \neq 0$ at every point of $M^n$, then each of the previous assertions is equivalent to: 
\begin{enumerate}
    \item[(4)] The normalized mean curvature vector field $\big(1/\sqrt{|\langle \mathbf{H},\mathbf{H} \rangle|}\, \big)\,\mathbf{H}$ is parallel.
\end{enumerate}
\end{proposition}
\begin{proof}
The equivalence between the first three points is a direct consequence of formulas (\ref{2805A}) and (\ref{mean_curvature_2}) since $\xi$ and $\eta$ are parallel for the normal connection. 

It is a general fact that under the assumption $\langle \mathbf{H}, \mathbf{H} \rangle \neq 0$ at every point of $M^n$, the condition $\nabla^{\perp} \mathbf{H}=0$ implies that the normalized mean curvature vector field is also parallel. 
Conversely, if $\langle \mathbf{H},\mathbf{H} \rangle$ never vanishes, from (\ref{2805A}) the normalized mean curvature vector field of $\psi$ 
is given by
$$
\frac{\sqrt{n(n-1)}}{\sqrt{|S|}}\, \mathbf{H}.
$$
Taking into account again that $\nabla^{\perp}\xi = \nabla^{\perp}\eta =0$ and (\ref{mean_curvature_2}), we get 
$$
X(|S|)=0,
$$
for any $X\in\mathfrak{X}(M^n)$ and consequently $S$ is constant.
\end{proof}

\begin{remark}\label{final_2.1}{\rm 

 Observe that for any spacelike submanifold $\psi : M^n \rightarrow \L ^{n+2}$ such that $\psi(M^n)\subset \Lambda^{n+1}_+$ we have that $\mathbf{H}$ never vanishes thanks to (\ref{mean_curvature_2}). However, $\langle \mathbf{H},\mathbf{H} \rangle$ may vanish somewhere. 
Even more, in the non-compact case, there is an isometric embedding given by
$$
\psi\colon \E^{n}\rightarrow \Lambda^{n+1}_+\subset \L^{n+2}, \quad \psi(x)=\Big(\frac{1+|x|^2}{2}, \frac{-1+ |x|^2}{2}, x\Big),
$$
where $|\cdot|$ denotes the usual Euclidean norm in $\E^n$, \cite{Palmer}, whose mean curvature vector field satisfies $\mathbf{H}=-\eta$, directly from (\ref{mean_curvature_2}), and, hence $\langle \mathbf{H},\mathbf{H} \rangle=0$ everywhere.
On the other hand, $\psi_0(x)=(1+|x|^2)/2$, and (\ref{Weingarten_operators}) gives 
$A_{\eta}=0$ at $x\in \mathbb{E}^n$. Thus, this isometric embedding of $\E^{n}$ in $\mathbb{L}^{n+2}$ is totally umbilical. This situation contrasts with Counter-example \ref{counter-example} below, where we exhibit an isometric immersion of the Euclidean plane $\E^2$ in $\mathbb{L}^4$ through $\Lambda_+^3$ that is not totally umbilical.
}
\end{remark}

\begin{remark} {\rm
It should be recalled that every compact $n$-dimensional spacelike submanifold in $\L^{n+2}$ that factors through $\Lambda^{n+1}_+$ is diffeomorphic to an $n$-sphere \cite[Prop. 5.1, Remark 5.2]{PaRo}. Consequently, any compact $n$-dimensional spacelike submanifold in $\L^{n+2}$ that factors through $\Lambda^{n+1}_+$ must be necessarily embedded. However, that is not the case if the codimension of the spacelike submanifold is assumed to be $\geq 3$.
}
\end{remark}

\begin{example}\label{examples}{\rm (\cite[Ex. 4.2]{PaRo}, \cite[Th. 4.2]{ACR}) The totally umbilical compact spacelike submanifolds  $\psi : M^n \rightarrow \L ^{n+2}$ with
$\psi(M^n)\subset \Lambda^{n+1}_+$ ($M^n$ is diffeomorphic to an $n$-sphere from the previous Remark) can be explicitly described as follows: 
For each $v=(v_0,v_1,...,v_{n+1})\in \L^{n+2}$ such that $\langle v,v\rangle = -1$ and $v_0<0$, and each real positive number $r$, we put 
\begin{equation}\label{sphere}
\S^n(v,r)=\{x\in \L^{n+2}\, :\, \langle x,x\rangle = 0, \langle v,x\rangle = r\}.
\end{equation}
Every $\S^n(v,r)$ is a compact $n$-dimensional spacelike submanifold of $\L^{n+2}$ contained in $\Lambda^{n+1}_+$. Its normal bundle is spanned by $\xi, \eta \in \mathfrak{X}^{\perp}(\S^n(v,r))$, where  $\xi_x=x$ and $\eta_x=(1/2r^2)\,x+(1/r)\,v$, for any $x\in \S^n(v,r)$. We clearly have $\langle \xi,\xi \rangle = \langle \eta,\eta \rangle = 0$ and $\langle \xi,\eta \rangle = 1$. The corresponding Weingarten operators are 
$$
A_{\xi}=-\mathrm{Id},\quad   A_{\eta}=-\frac {1}{2r^2}\,\mathrm{Id}\,.
$$
Therefore $\S^n(v,r)$ is totally umbilical with $\mathrm{II}(X,Y)=\langle X,Y \rangle \mathbf{H}$, where $\mathbf{H}=-(1/2r^2)\,\xi - \eta$. Thus, from the Gauss equation (\ref{1}), we conclude that $\S^n(v,r)$ has constant sectional curvature equal to $1/r^2$. Summing up, $\S^n(v,r)$ defines a totally umbilical spacelike embedding of the $n$-dimensional sphere, of radius $r$, in $\L^{n+2}$ through the light cone $\Lambda^{n+1}_+$. Conversely, every totally umbilical codimension two compact spacelike submanifold $\in \mathbb{L}^{n+2}$ through $\Lambda^{n+1}_+$ can be achieved in this way (compare with \cite[Example 4.2]{PaRo}).
Note also that the map 
\begin{equation}\label{parametrization}
\S^{n}\rightarrow \Lambda^{n+1}_+, \quad x\mapsto \frac{r}{\langle v , (1,x)\rangle}\,(1,x) 
\end{equation}
provides a parametrization of each spacelike submanifold $\S^n(v,r)$.}
\end{example}

\section{Spacelike graphs in the light cone}

\noindent The sphere $\S^{n}$ can be topologically identified with the space of rays spanned by vectors in $\Lambda^{n+1}_{+}$, as a subset of the real projective space $\R \textrm{P}^{n+1}$. Explicitly, we have a natural projection
\begin{equation}\label{projection}
\pi \colon \Lambda^{n+1}_{+} \rightarrow \S^{n},\quad (y_0,y_1,\cdots ,y_{n+1}) \mapsto \frac{1}{y_{0}}\,(y_{1}, \cdots , y_{n+1}).
\end{equation}
The smooth map $\pi$ is a trivial principal bundle with structure group $\R_{>0}$, the multiplicative group of positive real numbers, and its vertical distribution agrees with the radical of the degenerate metric on $\Lambda^{n+1}_{+}$. Every global section of $\pi$ provides a global trivialization of $\Lambda^{n+1}_{+}$. In particular, the global section 
\begin{equation}\label{global_section}
\S^{n}\to \Lambda^{n+1}_{+},\quad x \mapsto (1,x),
\end{equation}
provides us with the diffeomorphism 
\begin{equation}\label{associated_diffeomorphism}
\Phi\colon \R_{>0}\times \S^{n} \to \Lambda^{n+1}_{+}, \quad \Phi(t,x)=(t,tx).
\end{equation}
More generally, there is a one-to-one correspondence between $C^{\infty}(\S^{n})$ and the set of all global sections of $\pi$, $\Gamma(\pi)$, defined as follows: for every $f\in C^{\infty}(\S^{n})$, we can consider the section of $\pi$ given by
\begin{equation}\label{global_section2}
i_{f}\colon \S^{n}\to \Lambda^{n+1}_{+},\quad x \mapsto e^{f(x)}(1,x).
\end{equation}
This assignment defines a one-to-one map $i\colon C^{\infty}(\S^{n}) \to \Gamma(\pi)$, with inverse
$$
i^{-1}(s)(x)=\log (s_{0}(x)),\quad x\in \S^{n},
$$
where $s\in \Gamma(\pi)$ is given by $s(x)=(s_{0}(x),s_1(x), \cdots , s_{n+1}(x)).$
Note that $i_{f}$ corresponds to the global trivialization $(t,x)\mapsto e^{f(x)}(t,tx)$ of $\Lambda^{n+1}_{+}$. 

\smallskip

Every section $i_f$ of $\pi$ can be considered a codimension two spacelike embedding 
and the induced metric $g_f$ on $\S^n$  via $i_f$ satisfies 
\begin{equation}\label{induced_metric}
g_f= e^{2f}g_{0}, 
\end{equation}
where $g_{0}$ is the round metric of constant sectional curvature $1$. Moreover, every Riemannian metric on $\S^n$ conformally related to $g_{0}$ is achieved in this way (Remark \ref{final_2.1}).  In conformal geometry, this fact motivates that $\Lambda^{n+1}_{+}$ is called the bundle of scales of $\S^{n}$.  

\vspace{1mm}

In this picture, the orthochronous Lorentz group $O^{+}(1, n+1)$ acts transitively on $\S^{n}$, preserving the conformal class of the round metric. The sphere $\S^{n}$, as a homogeneous space of $O^{+}(1, n+1)$, is the model for the conformal Riemannian geometry as a Cartan geometry, see for instance, \cite[Section 1.6]{CS09}.

\smallskip

This construction motives that, for every $f\in C^{\infty}(\S^n)$, the hypersurface 
\begin{equation}\label{grafo}
\Sigma^{n}_{f}:=\big\{\,(e^{f(x)}, e^{f(x)}x)\in \Lambda^{n+1}_{+}\; :\; x\in \S^{n}\,\big\}
\end{equation}
in $\Lambda^{n+1}_+$ may be called the graph in $\Lambda^{n+1}_{+}$ defined by $e^{f}$. Clearly, every graph $\Sigma^{n}_{f}$ is a compact codimension two spacelike submanifold in $\L^{n+2}$ through $\Lambda^{n+1}_{+}$.

\begin{proposition}\label{010325A}
Each spacelike embedding $\psi\colon\S^{n}\to \Lambda^{n+1}_{+}$ may be written as  
$$
\psi(x)=e^{f(x)}(1, \Phi(x)),\quad x\in \S^{n},
$$
where $f:=\log\psi_{0}$ and $\Phi$  is the diffeomorphism of $\mathbb{S}^n$ given by 
$$\Phi(x)=\Big(\,\frac{\psi_{1}(x)}{e^{f(x)}}, \cdots , \frac{\psi_{n}(x)}{e^{f(x)}}\,\Big).
$$ Thus, we have $\psi=i_{f}\circ \Phi$, that is, up to a diffeomorphism of $\S^{n}$, any compact codimension two spacelike submanifold in $\L^{n+2}$ contained in $\Lambda^{n+1}_{+}$ is a graph over $\mathbb{S}^n$ in $\Lambda^{n+1}_{+}$. In this case, the metric $g$ induced on $\S^n$ via $\psi$ satisfies $g=e^{2f\circ \Phi}\,\Phi^{*}(g_{0}).$ 
\end{proposition}
\begin{proof}
The decomposition $\psi=i_{f}\circ \Phi$ is trivial, and the assertion of the induced metric is a direct consequence of this decomposition; it only remains to show that $\Phi$ is a diffeomorphism. In fact, since
$$
(d\psi)_x (v)= v(f \circ \Phi)\,e^{f(\Phi(x))}\,(1, \Phi(x))+\,e^{f(\Phi(x))}\,(0, (d\Phi)_x(v)),
$$
we have $\Phi\colon \S^{n}\to \S^{n}$ is a local diffeomorphism. But a local diffeomorphism between compact and simply connected manifolds must be a global diffeomorphism \cite[Prop. 5.6.1]{DoCarmo}. 
\end{proof}
Using the relation between the scalar curvatures of two conformally related metrics, we get for the scalar curvature $S_{f}$ of $g_{f}=e^{2f}g_{0}$,
\begin{equation}\label{10025A}
  e^{2f}S_{f}=n(n-1)-2(n-1)\Delta^{0}\, f- (n-1)(n-2)\| \nabla^{0} f \|_{0}^2,   
\end{equation}
where $\Delta^0$, $\nabla^{0}$ and $\| \; \|_{0}^2$ denote the Laplacian, the gradient and the square of the norm for the metric $g_{0}$, respectively.
In particular, using (\ref{2805A}), we have for the square of the length of the mean curvature vector field $\mathbf{H}_{f}$ of $i_{f}$ the following formula:
\begin{equation}\label{0403A}
\langle \mathbf{H}_{f}, \mathbf{H}_{f}\rangle=\frac{1}{n\, e^{2f}}\Big( n-2\Delta^{0}f-(n-2)\|\nabla^{0}f \|^{2}_{0}\Big).
\end{equation}

\begin{example}{\rm As a complement to Remark \ref{1908B}, we will show in this example a spacelike graph over $\S^n$ in $\Lambda^{n+1}_+$, such that the square length of the mean curvature vector field can take negative and positive values.  
Choose now $f$ as the restriction on $\S^n$ of the projection $(x_{1},\cdots ,x_{n+1}) \mapsto x_{n+1}$. Because $f$ is the restriction to $\S^n$ of a linear function, we have  $\Delta^{0} f + nf=0$, where $\Delta^{0}$ denotes the Laplacian of the metric $g_{0}$. That is, $f$ is an eigenfunction of the Laplacian corresponding to $\lambda_{1}(\S^{n},g_{0})=n$. A direct computation gives that
$
\| (\nabla^{0} f) (x) \|_{0}^2=1-x^{2}_{n+1}$, at every $x\in \S^{n}$. Hence, from formulas (\ref{2805A}) and (\ref{10025A}), we obtain
\begin{equation}\label{130125A}
\langle\mathbf{H}_f,  \mathbf{H}_f\rangle (x)=\frac{1}{n(n-1)}\,S_{f}(x)=\frac{1}{n\, e^{2x_{n+1}}}\Big( (n-2)\,x^{2}_{n+1}+ 2n \,x_{n+1} +2\Big),
\end{equation}
where $\mathbf{H}_f$ is the mean curvature vector field of $i_f$. Thus, we have
$$
\langle\mathbf{H}_f,  \mathbf{H}_f\rangle_{(0, \cdots, 1)}=\frac{3}{e^2}, \quad \textrm{ and }\quad \langle\mathbf{H}_f,  \mathbf{H}_f\rangle_{(0, \cdots, -1)}=-e^2.
$$
Therefore, there exists some point where $\langle\mathbf{H}_{f},  \mathbf{H}_{f}\rangle$ vanishes. In the particular case $n=2$, formula (\ref{130125A}) reduces to
$$
\langle\mathbf{H}_f,  \mathbf{H}_f\rangle (x)= K_{f}(x)=\frac{2x_{3}+1}{e^{2x_{3}}},
$$ 
where $K_f$ is the Gauss curvature of the metric $g_{f}$ (Remark \ref{n=2}). Therefore, $\langle\mathbf{H}_f,  \mathbf{H}_f\rangle (x)=0$, at any point $x=(x_{1},x_{2}, -1/2)\in \S^{2}$.
 } 
\end{example}

\smallskip

Coming back to the general case, the square length of the mean curvature vector field $\mathbf{H}_f$ of $\Sigma^{n}_{f}$ in $\L^{n+2}$ satisfies
\begin{equation}\label{PDE}
2\Delta^{0}f+(n-2)\,\|\nabla^{0}f\|^{2}_{0}=n\,(1- \|\mathbf{H}_{f}\|^{2}\,e^{2f}),
\end{equation}
where we put $\|\mathbf{H}_f\|^{2}=\langle \mathbf{H}_f,\mathbf{H}_f\rangle$, and $\nabla^0$ and $\Delta^{0}$ are the gradient and the Laplacian on $(\S^{n}, g_{0})$, respectively, see (\ref{0403A}).
From Remark \ref{110725A}, if we assume $ \|\mathbf{H}_f\|^{2}$ constant,  then $ \|\mathbf{H}_f\|^{2}$ is a positive constant $k$. Recall also that the assumption  $ \|\mathbf{H}_f\|^{2}$ constant turns out into $\nabla^{\perp}\mathbf{H}_f=0$, where $\nabla^{\perp}$ is the normal connection, according to Proposition \ref{equivalence2}.
Therefore, for each constant $k>0$, the elliptic PDE  
\[
(\mathrm{E}) \hspace*{40mm} 2\Delta^{0}f+(n-2)\,\|\nabla^{0}f\|^{2}_{0}=n\,(1- k\,e^{2f})\hspace*{50mm} 
\]
is the equation of spacelike graphs on $\S^{n}$ in $\Lambda^{n+1}_{+}$ with $\|\mathbf{H}_f\|^{2}=k$; i.e., spacelike graphs on $\S^{n}$ in $\Lambda^{n+1}_{+}$ with parallel mean curvature vector field in $\L^{n+2}.$

\vspace{1mm}
 
Thanks to Proposition \ref{equivalence2} and formula (\ref{2805A}), for each $k>0$, equation (E) admits two distinct variational interpretations. On the one hand, since (E) is equivalent to the condition that $g_{f}$ has constant scalar curvature $n(n-1)k$, $k>0$, it can be seen as the Euler–Lagrange equation corresponding to the variational problem
\[
h\in C^{\infty}(\S^n) \longmapsto \int_{\S^n}S_h\,dV_{g_h}\in\mathbb{R},
\]
where $S_h$ denotes the scalar curvature of the metric $g_{h}=e^{2h}g_0$, under the volume constraint
\[
\int_{\S^n}dV_{g_h}=\frac{1}{\sqrt{k}}\mathrm{vol}(\S^n),
\]
where $\mathrm{vol}(\S^n)$ is the volume of $(\S^{n}, g_{0})$, see  \cite[Prop. 4.25]{Besse}.

\vspace{1mm}

On the other hand, since (E) is also equivalent to the condition that $i_{f}$ has parallel mean curvature vector field, it can alternatively be viewed as the Euler–Lagrange equation of the variational problem
\[
h\in C^{\infty}(\S^n) \longmapsto \int_{\S^n}\|\mathbf{H}_h\|^2\,dV_{g_h}\in\mathbb{R},
\]
where $\mathbf{H}_h$ denotes the mean curvature vector field of the space-like embedding $i_h : \S^n \mapsto \Lambda^{n+1}_+ \subset \L^{n+2}$ as defined in (\ref{global_section2}), under the same volume constraint as above. From (\ref{2805A}), the two variational problems are equivalent.

\vspace{1mm}

If we assume $n\geq 3$ and make the change of variable $e^{2f}=\varphi^{\frac{4}{n-2}}$, $\varphi>0$, differential equation (E) turns into 
\[
\mathrm{(E)'} \hspace*{50mm} \square^0 \varphi +\frac{n(n-2)}{4}\, k\, \varphi^{\frac{n+2}{n-2}}=0,\hspace*{55mm}
\]
where $\square^0 \varphi:=\Delta^0 \varphi - (n(n-2)/4)\varphi$, which is a very special case of the Yamabe equation \cite[eq. (1.2)]{Lee-Parker} (note that the convention sign for the Laplacian in \cite{Lee-Parker} is the opposite to the one considered here). Of course, if $\|\mathbf{H}_f\|^2=k$ for some $f\in C^{\infty}(\mathbb{S}^n)$, then $\varphi$, defined as previously from $f$, satisfies $\mathrm{(E)'}$.

\vspace{1mm}

Following \cite[p. 39]{Lee-Parker} we put $\|\varphi\|_p=\Big(\int_{\S^n}\varphi^p\,dV_{g_0}\Big)^{\frac{1}{p}}$, where $p=2n/(n-2)$, and $E(\varphi)=\int_{\S^n}\big(\,\|\nabla^0\varphi \|^2_0 + (n(n-2)k/4)\varphi^2\,\big)dV_{g_0}$. For $e^f=\varphi^{\frac{2}{n-2}}$, we have
\[
\|\varphi\|^2_p=\Big(\int_{\S^n}dV_{g_f}\Big)^{\frac{2}{p}}\quad \text{and} \quad \frac{4}{n(n-2)}E(\varphi)= \int_{\S^n}\|\mathbf{H}_f\|^2\,dV_{g_f}\,.
\]
Again from \cite[p. 39]{Lee-Parker}, equation $\mathrm{(E)'}$ has a variational meaning for each $k>0$. It is the Euler-Lagrange equation for the variational problem
\[
\varphi\in C^{\infty}(\S^n),\, \varphi>0\, \longmapsto \, \frac{E(\varphi)}{\|\varphi \|^2_p}\in\mathbb{R}.
\]
\section{Integral formulas}
\noindent Throughout this section, let $\psi: M^{n}\rightarrow \L^{n+2}$ be a spacelike submanifold and $\xi, \eta \in \mathfrak{X}^{\perp}(M^n)$ such that $\langle \xi, \xi\rangle=\langle \eta, \eta\rangle =0$ and $\langle \xi, \eta\rangle =1$.
For each $a\in \L^{n+2}$ consider $a^{\top}\in \mathfrak{X}(M^n)$  and $a^{\mathrm{N}}\in \mathfrak{X}^{\perp}(M^n)$ its tangential and normal components to $M^n$, respectively, defined as in (\ref{decomposition_1}).  It is easy to see that 
\begin{equation}\label{gradiente}
a^{\top}=\nabla\langle a, \psi \rangle,
\end{equation} 
where $\nabla$ denotes the gradient operator on $M^n$, and 
\begin{equation}\label{normal_part}
a^{\mathrm{N}}= \langle a, \xi \rangle \eta + \langle a, \eta \rangle \xi.
\end{equation}
Now, from $\overline{\nabla}a=0$, using (\ref{Gauss_Weingarten}) and taken into account (\ref{normal_part}) jointly with (\ref{one-form}) we get
\begin{equation}\label{2505}
\nabla_{X}a^{\top}-\langle a, \xi\rangle A_{\eta}(X)-\langle a, \eta\rangle A_{\xi}(X)=0,
\end{equation}
and 
\begin{equation}\label{2505_a}
\mathrm{II}(X, a^{\top})-\langle a, A_{\xi}(X)\rangle \eta-\langle a, A_{\eta}(X)\rangle \xi=0,
\end{equation}
for any $X\in \mathfrak{X}(M)$. 
\begin{proposition}\label{2605C}
Let $\psi: M^{n}\rightarrow \L^{n+2}$ be a compact spacelike submanifold and let us consider $\xi, \eta \in \mathfrak{X}^{\perp}(M^n)$ such that $\langle \xi, \xi\rangle=\langle \eta, \eta\rangle =0$ and $\langle \xi, \eta\rangle =1$. For each $a\in \L^{n+2}$, the following integral formula holds
$$
\int_{M^n}\Big[\frac{n-1}{n}\,a^{\top}(\mathrm{trace}(A_{\eta}))+ \langle a, \xi\rangle\Big(\mathrm{trace}(A^{2}_{\eta})-\frac{1}{n}\,(\mathrm{trace}(A_{\eta}))^2\Big)
$$
$$
 +\langle a, \eta\rangle \Big(\mathrm{trace}(A_{\eta} A_{\xi})-\frac{1}{n}\,\mathrm{trace}(A_{\eta})\mathrm{trace}(A_{\xi})\Big)
$$   
\begin{equation}\label{030125A}
\hspace*{-6mm}-\langle A_{\eta}a^{\top},\alpha^{\sharp}\rangle + \langle a^{\top}, \alpha^{\sharp} \rangle \,\mathrm{trace}(A_{\eta})\Big]\,dV_g=0,
\end{equation}
where $\alpha^{\sharp}\in \mathfrak{X}(M^n)$ is metrically equivalent to the one-form $\alpha$ given in {\rm (\ref{one-form})}.
\end{proposition}
\begin{proof}
Directly from (\ref{2505}), we get the following formula for the divergence of the vector field $a^{\top}$,
\begin{equation}\label{2505B}
\mathrm{div}(a^{\top})=\langle a, \xi \rangle \mathrm{trace}(A_{\eta})+\langle a, \eta\rangle \mathrm{trace}(A_{\xi}).
\end{equation}
Indeed, the previous formula is an equivalent way to write Beltrami's equation $\triangle \langle a, \psi \rangle =n \langle a, \mathbf{H}\rangle$, taking into account (\ref{curme}) and (\ref{gradiente}).

\vspace{1mm}

Now, fix a point $p\in M^n$ and a local orthonormal frame $(E_{1}, \cdots, E_{n})$ in $M^n$ around $p$ with $\nabla E_{i}=0$ at $p$. Using  the Codazzi equation (\ref{Codazzi_2}) and formula (\ref{2505}), a direct computation shows that, at the point $p\in M^n$, we have 
\begin{align*}
    \mathrm{div}(A_{\eta}a^{\top})&=\sum_{i=1}^{n}\langle \nabla_{E_{i}}A_{\eta}a^{\top}, E_{i}\rangle
\\ &=\sum_{i=1}^{n}\big\langle (\nabla_{E_{i}}A_{\eta})a^{\top}+A_{\eta}(\nabla_{E_{i}}a^{\top}), E_{i}\big\rangle \\
&=\sum_{i=1}^{n}\big\langle (\nabla_{a^{\top}}A_{\eta})E_{i}-\alpha(E_{i})A_{\eta}(a^{\top})+\alpha(a^{\top})A_{\eta}(E_{i})+A_{\eta}(\nabla_{E_{i}}a^{\top}), E_{i}\big\rangle \\
&
=\sum_{i=1}^{n}\big\langle \nabla_{a^{\top}}(A_{\eta}E_{i})+A_{\eta}(\nabla_{E_{i}}a^{\top}), E_{i}\big\rangle -\langle A_{\eta}a^{\top} ,\alpha^{\sharp} \rangle +\alpha(a^{\top})\mathrm{trace}(A_{\eta}) \\
&
=a^{\top}(\mathrm{trace}(A_{\eta}))-\langle A_{\eta}a^{\top},\alpha^{\sharp}\rangle + \langle a^{\top}, \alpha^{\sharp} \rangle \,\mathrm{trace}(A_{\eta}) \\[2mm] & + \langle a, \xi\rangle \,\mathrm{trace}(A^2_{\eta})+\langle a, \eta\rangle \,\mathrm{trace}(A_{\xi}A_{\eta}).
\end{align*}

\noindent Using now the previous formula and (\ref{2505B}), we obtain
\begin{align*}
\mathrm{div}\Big(A_{\eta}a^{\top}-\frac{1}{n}\,\mathrm{trace}(A_{\eta}) a^{\top}\Big) 
&=\frac{n-1}{n}a^{\top}(\mathrm{trace}(A_{\eta}))\\& +  \langle a, \xi\rangle\Big(\mathrm{trace}(A^{2}_{\eta})-\frac{1}{n}\,(\mathrm{trace}(A_{\eta}))^2\Big)\\ &+\langle a, \eta\rangle \Big(\mathrm{trace}(A_{\xi} A_{\eta})-\frac{1}{n}\,\mathrm{trace}(A_{\xi})\,\mathrm{trace}(A_{\eta})\Big)\\[2mm] &-\langle A_{\eta}a^{\top},\alpha^{\sharp}\rangle + \langle a^{\top}, \alpha^{\sharp} \rangle \,\mathrm{trace}(A_{\eta}).
\end{align*}
The result then follows by integrating the previous formula over $M^n$.
\end{proof}

\begin{remark}{\rm  In Proposition \ref{2605C}, the roles of $\xi$ and $\eta$ can be interchanged. Indeed, applying the same argument as before yields another integral formula analogous to (\ref{030125A}):
$$
\int_{M^n}\Big[\frac{n-1}{n}\,a^{\top}(\mathrm{trace}(A_{\xi}))+ \langle a, \eta\rangle\Big(\mathrm{trace}(A^{2}_{\xi})-\frac{1}{n}\,(\mathrm{trace}(A_{\xi}))^2\Big)
$$
$$
 +\langle a, \xi\rangle \Big(\mathrm{trace}(A_{\eta} A_{\xi})-\frac{1}{n}\,\mathrm{trace}(A_{\eta})\mathrm{trace}(A_{\xi})\Big)
$$   
\begin{equation}
\hspace*{-6mm}+\langle A_{\xi}a^{\top},\alpha^{\sharp}\rangle - \langle a^{\top}, \alpha^{\sharp} \rangle \,\mathrm{trace}(A_{\xi})\Big]\,dV_g=0.
\end{equation}
}
\end{remark}

\begin{corollary}\label{111024a}
Let $\psi: M^{n}\rightarrow \L^{n+2}$ be a compact spacelike submanifold such that $\nabla^{\perp} \mathbf{H}=0$. Suppose $\xi_0, \xi_1 \in \mathfrak{X}^{\perp}(M^n)$, such that $\langle \xi_0, \xi_0\rangle=\langle \xi_1, \xi_1\rangle =0$ and $\langle \xi_0, \xi_1\rangle =1$ are taken as in Corollary $\ref{coro1}$. For each $a\in \L^{n+2}$, the following integral formulas hold
$$
\int_{M^n}\Big[\langle a, \xi_i\rangle\Big(\mathrm{trace}(A^{2}_{\xi_{i+1}})-\frac{1}{n}\,(\mathrm{trace}(A_{\xi_{i+1}}))^2\Big)
$$
\begin{equation}\label{111024b}
 +\langle a, \xi_{i+1}\rangle \Big(\mathrm{trace}(A_{\xi_{i+1}} A_{\xi_{i}})-\frac{1}{n}\,\mathrm{trace}(A_{\xi_{i+1}})\mathrm{trace}(A_{\xi_{i}})\Big)\Big]\,dV_g=0\,,
\end{equation}   
where $i\in \mathbb{Z}_2$. Therefore, if $\xi_{i}$ is umbilical for some $i\in \mathbb{Z}_{2}$, then $\psi$ is totally umbilical. 
\end{corollary}
\begin{proof}
From Corollary \ref{coro1}, formula (\ref{030125A}) directly reduces to (\ref{111024b}). On the other hand, if $\xi_i$ is assumed to be umbilical, formula (\ref{111024b}) gives
$$
\int_{M^n}\langle a, \xi_i\rangle\Big(\mathrm{trace}(A^{2}_{\xi_{i+1}})-\frac{1}{n}\,(\mathrm{trace}(A_{\xi_{i+1}}))^2\Big)dV_g=0\,,
$$
for any vector $a\in \L ^{n+2}$. In particular, the vector $a$ can be chosen to be timelike, and then the function $\langle a, \xi_i\rangle$ never vanishes. Consequently, the previous formula implies that $\mathrm{trace}(A^{2}_{\xi_{i+1}})-\frac{1}{n}\,(\mathrm{trace}(A_{\xi_{i+1}}))^2=0$, ending the proof.
\end{proof}
For each submanifold $M^n$ in Euclidean space $\E^m$, $m>n$, its mean curvature vector field $\mathbf{H}$ and the scalar curvature $S$ of $M^n$ are related by the inequality $S\leq n(n-1)\langle \mathbf{H},\mathbf{H} \rangle$, and the equality holds if and only if $M^n$ is totally umbilical. In the case of an $n$-dimensional spacelike submanifold in $\L^{n+p}$, $p\geq 2$, this inequality does not hold, in general. However, we have,

\begin{proposition}\label{040125B}
  Let $\psi: M^{n}\rightarrow \L^{n+2}$ be a compact spacelike submanifold with $\nabla^{\perp} \mathbf{H}=0$.
  Then, for every $a\in \L^{n+2}$ timelike with $\langle a,\xi_{0}\rangle >0$, where $\xi_{0}$ and $\xi_1$ are taken as in Corollary $\ref{coro1}$, the following inequality holds
\begin{equation}\label{040125A}
    \int_{M^n}\langle a, \xi_{0}-\xi_{1}\rangle \Big(n(n-1)\langle \mathbf{H},\mathbf{H} \rangle- S\Big)\,dV_g \, \geq 0\,.
\end{equation}
The equality holds if and only if  $M^n$ is a totally umbilical round sphere with constant sectional curvature $\langle \mathbf{H},\mathbf{H} \rangle >0$, and contained in a spacelike hyperplane of $\L^{n+2}$.
\end{proposition}
\begin{proof}
    Taking into account (\ref{squared_lenght_H}) and (\ref{101}), we can write
    \begin{equation}\label{140125A}
    \mathrm{trace}(A_{\xi_{0}} A_{\xi_{1}})-\frac{1}{n}\,\mathrm{trace}(A_{\xi_{0}})\mathrm{trace}(A_{\xi_{1}})=\frac{1}{2}\,\big(n(n-1)\langle \mathbf{H},\mathbf{H}\rangle - S\big).    
    \end{equation}
    Therefore, Corollary \ref{111024a} implies that
    $$
 \int_{M^n}\langle a, \xi_{0}\rangle \Big(n(n-1)\langle \mathbf{H},\mathbf{H} \rangle- S\Big)\,dV_g =-2\int_{M^n} \langle a, \xi_1\rangle\Big(\mathrm{trace}(A^{2}_{\xi_{0}})-\frac{1}{n}\,(\mathrm{trace}(A_{\xi_{0}}))^2\Big)dV_g\geq 0\,,
 $$
 where we have used that $\langle a, \xi_{1}\rangle <0$ holds because $\langle \xi_0,\xi_1\rangle = 1>0$. Similarly we get
 $$
\int_{M^n}\langle a, -\xi_{1}\rangle \Big(n(n-1)\langle \mathbf{H},\mathbf{H} \rangle- S\Big)\,dV_g =2\int_{M^n} \langle a, \xi_0 \rangle\Big(\mathrm{trace}(A^{2}_{\xi_{1}})-\frac{1}{n}\,(\mathrm{trace}(A_{\xi_{1}}))^2\Big)dV_g\geq 0\,.
$$
This concludes the proof of the inequality. Now, the equality holds in (\ref{040125A}) if and only if 
\begin{equation}\label{Z_2}
\int_{M^n}\langle a, \xi_i\rangle\Big(\mathrm{trace}(A^{2}_{\xi_{i+1}})-\frac{1}{n}\,\mathrm{trace}(A_{\xi_{i+1}})^2\Big)dV_g=0, \quad i\in \Z_{2},
\end{equation}
that is, if and only if $\psi: M^{n}\rightarrow \L^{n+2}$ is totally umbilical. 

\vspace{1mm}

If $\psi$ is totally umbilical then $n(n-1)\langle \mathbf{H},\mathbf{H} \rangle= S$ holds without any additional assumptions.
Indeed, from the umbilicity condition, we have $\langle \Pi,\Pi \rangle =n\langle \mathbf{H}, \mathbf{H} \rangle$, which, when substituted in (\ref{101a}), gives the announced equality (alternatively, this can be also deduced directly from (\ref{140125A})). On the other hand, if the equality $n(n-1)\langle \mathbf{H},\mathbf{H} \rangle= S$ holds, then under our assumptions we arrive again to (\ref{Z_2}), and thus $\psi$ is totally umbilical. 

\vspace{1mm}

In addition, note that the constant $\langle \mathbf{H},\mathbf{H} \rangle$ must be positive (see \cite[Prop. 4.1]{AER}). From (\ref{1}) we have that the sectional curvature of $M^n$ is constant and equals $\langle \mathbf{H},\mathbf{H} \rangle >0$. Finally, if $A_{\xi_{0}}=\lambda_{0}\mathrm{Id}$ and $A_{\xi_{1}}=\lambda_{1}\mathrm{Id}$, $\lambda_{0},\lambda_{1}\in \mathbb{R}$, then (\ref{squared_lenght_H}) implies that the vector $\lambda_{1}\xi_{0}-\lambda_{0}\xi_{1}$ is timelike. Thus,
$\psi(M^n)$ is contained in an affine spacelike hyperplane $H$ with normal vector $\lambda_{1}\xi_{0}-\lambda_{0}\xi_{1}$, (alternatively see, for instance, \cite[Prop. 4.3]{Da}). Thus $\psi(M^n)$ is a round sphere of radius $1/\sqrt{\langle \mathbf{H},\mathbf{H} \rangle }$ in $H \cong \E^{n+1}$. The converse is straightforward.
\end{proof}

\begin{remark}\label{Wente} {\rm There exist codimension two compact spacelike submanifolds in Lorentz-Minkowski spacetime with parallel mean curvature vector field that are not totally umbilical round spheres. For instance, the well-known Wente torus in $\E^3$, \cite{Wente}, naturally gives rise to an immersed compact spacelike surface in $\L^4$ with parallel mean curvature vector field. However, the only topological $2$-spheres that can be immersed as spacelike surfaces in $\L^4$ with parallel (necessarily spacelike) mean curvature vector field are the totally umbilical round spheres lying in a spacelike affine hyperplane of $\L^4$, \cite[Cor. 3.4]{AER}.
}
\end{remark}

\section{Main results}

\noindent We begin this section with the following immediate consequence of Proposition \ref{040125B},

\begin{corollary}
  Let $\psi: M^{n}\rightarrow \L^{n+2}$ be a compact spacelike submanifold with $\nabla^{\perp} \mathbf{H}=0$.
  Assume $S\geq n(n-1)\langle \mathbf{H},\mathbf{H} \rangle$, then $M^n$ is a totally umbilical round sphere with constant sectional curvature $\langle \mathbf{H},\mathbf{H} \rangle >0$, and contained in a spacelike hyperplane of $\L^{n+2}$.
\end{corollary}

\begin{remark}\label{wrong_inequality}{\rm For any submanifold $M^n$ in the Euclidean space $\E^m$, $m>n$, we have $S=n(n-1)\langle \mathbf{H},\mathbf{H}\rangle +n\langle \mathbf{H},\mathbf{H}\rangle-\langle \Pi,\Pi \rangle \leq n(n-1)\langle \mathbf{H},\mathbf{H}\rangle$, because 
$n\langle \mathbf{H},\mathbf{H}\rangle\leq \langle \Pi,\Pi \rangle$ by the Schwarz inequality for self-adjoint operators of a Euclidean vector space, and the equality holds if and only if $M^n$ is totally umbilical. Obviously, there are serious algebraic reasons that prevent an extension of this fact for spacelike submanifolds $M^n$ in $\mathbb{L}^{n+2}$, $p\geq 2$. 
}
\end{remark}

As a direct consequence of Proposition \ref{040125B} and formula (\ref{2805A}) we get,
\begin{theorem}\label{1106A}
Let $\psi : M^{n}\rightarrow \L^{n+2}$ be a compact spacelike submanifold such that 
$\psi(M^n)\subset \Lambda^{n+1}_+$, with parallel mean curvature vector field $\mathbf{H}$. Then $M^n$ is a totally umbilical round sphere obtained as $\psi(M^n)=\S^n(v,r)$, where $v\in \mathbb{L}^{n+2}$ satisfies $\langle v,v \rangle =-1$, $v_0<0$ and $1/r^2=\langle \mathbf{H},\mathbf{H} \rangle$ is the sectional curvature of $M^n$. 
\end{theorem}

In the compact case, Proposition \ref{equivalence2} can be sharpened,

\begin{corollary}\label{equivalence3}
 Let $\psi : M^{n}\rightarrow \L^{n+2}$ be a compact spacelike submanifold such that 
$\psi(M^n)\subset \Lambda^{n+1}_+$. Then, the following assertions are equivalent:
 \begin{enumerate}
     \item The mean curvature vector field of $\psi$ is parallel,
     \item The second fundamental form of $\psi$ is parallel,
     \item The scalar curvature of $M^n$ is constant,
     \item The sectional curvature of $M^n$ is constant,  
     \item $M^n$ is a totally umbilical round sphere.
 \end{enumerate}
\end{corollary}
\begin{Counter-example}\label{counter-example}
{\rm {\bf (a)} Theorem \ref{1106A} cannot be extended to the case where $M^n$ is assumed to be complete and non-compact. To support this assertion, consider the isometric immersion 
\begin{equation}\label{260725B}
    \psi : \E^2\rightarrow \L^4, \quad \psi(x,y)=\big(\cosh x, \sinh x, \cos y, \sin y\big).
\end{equation}
Clearly, $\psi(\E^2)\subset \Lambda^{3}_+$, and hence $\langle \mathbf{H},\mathbf{H} \rangle=0$ from (\ref{2805A}). Now, Proposition \ref{equivalence2} does imply $\nabla^{\perp}\mathbf{H}=0$. On the other hand, a direct computation from (\ref{Weingarten_operators}) shows that
$A_{\eta}(\partial/\partial x)=\frac{1}{2}\, \partial/\partial x$ and  $A_{\eta}(\partial/\partial y)=-\frac{1}{2}\, \partial/\partial y.$
Thus, this complete spacelike surface is not totally umbilical in $\L^4$.

{\bf (b)} It is well known that every submanifold with parallel second fundamental form also has parallel mean curvature vector field, and that the converse is not true in general. However, Corollary \ref{equivalence3} establishes the equivalence between these two properties for compact $n$-dimensional spacelike submanifolds in $\Lambda^{n+1}_{+}$. This result cannot be weakened to the geodesically complete case. Namely, there exist complete spacelike submanifolds $M^n$ in $\Lambda^{n+1}$ with $\nabla^{\perp}\mathbf{H}=0$ but $\widetilde{\nabla}\Pi\neq 0$.
To construct a counter-example, observe that for a spacelike submanifold $\psi \colon M^n \to \Lambda^{n+1}_{+}\subset \mathbb{L}^{n+2}$, formula (\ref{segunda_forma_fundamental}) reduces in this case to 
$$
\mathrm{II}(X,Y)=\langle A_{\eta}X, Y \rangle\,\xi-\langle X, Y \rangle\,\eta,
$$
for every $X,Y\in \mathfrak{X}(M^{n})$. Therefore, taking into account that $\nabla^{\perp}\xi=\nabla^{\perp}\eta=0$, the condition  $\widetilde{\nabla}\Pi= 0$ is equivalent to $\nabla A_{\eta}=0.$

\vspace{1mm}

Now, for each $\sigma \in C^{\infty}(\R^2)$ depending only on $x$, consider the spacelike immersion $\psi_{\sigma}:=e^{\sigma}\psi$, where $\psi$ is as given in (\ref{260725B}). The corresponding induced metric is $g^{*}=e^{2\sigma}\langle\,,\, \rangle_{0}$, where $\langle\,,\rangle_{0}$ is the Euclidean metric on $\E^{2}$. The well-known formulas for the Levi-Civita connection of conformally related metrics \cite[1.159]{Besse} yield 
$$
\nabla^{*}_{\partial_{x}}\partial_{x}=\sigma_{x}\,\partial_{x},\quad \nabla^{*}_{\partial_{x}}\partial_{y}=\sigma_{x}\,\partial_{y},\quad\nabla^{*}_{\partial_{y}}\partial_{y}=-\sigma_{x}\,\partial_{x},
$$
where $\nabla^*$ the Levi-Civita connection of $g^*$. The Weingarten endomorphism $A_{\eta_{\sigma}}$ is computed in \cite[Sect. 4]{PaRo} giving its matrix
$$
\frac{1}{2}\begin{pmatrix} \sigma_x^2-2\sigma_{xx}-1 & 0 \\ 0 & -\sigma_x^2 +1\end{pmatrix},
$$
relative to the canonical coordinate basis. Now, we are now in position to compute
$$
\Big(\nabla^{*}_{\partial_{x}}A_{\eta_{\sigma}}\Big)(\partial_{y})=\nabla^{*}_{\partial_{x}}(A_{\eta_{\sigma}}(\partial_{y}))-A_{\eta_{\sigma}}(\nabla^{*}_{\partial_{x}}\partial_y)=-\sigma_{x}\sigma_{xx}\,\partial_{y}.
$$
Let us consider the upper half-plane $U=\{(x,y)\in \R^2: x>0\}$.
We can specialize the above formula for the case $$e^{\sigma(x)}=\frac{1}{x}, \quad x>0,$$ and then
$$
\Big(\nabla^{*}_{\partial_{x}}A_{\eta_{\sigma}}\Big)(\partial_{y})=\frac{1}{x^3}\partial_{y}.
$$
Therefore, for this choice of $\sigma$, the second fundamental form of $\psi_{\sigma} : U \rightarrow  \mathbb{L}^{n+2}$ is not parallel.
Finally, taking into account that the Gauss curvature of $g^{*}$ is 
$$
K^{*}=-\frac{\triangle^{0} \sigma}{e^{2\sigma}}=-\frac{\sigma_{xx}}{e^{2\sigma}}=-1,
$$
a direct application of Proposition \ref{equivalence2} shows that the mean curvature vector field of $\psi_{\sigma}$ must be parallel. Additionally, note that the induced metric $g^{*}$ from $\psi_{\sigma}$ on the half-plane $U$ coincides with the Poincaré metric on $U$, and therefore $(U,g^{*})$ is the hyperbolic plane, which is obviously geodesically complete.
}
\end{Counter-example}

Next, we will give several applications of Theorem \ref{1106A}. 
First, recall that there is an explicit one-to-one correspondence between $n$-dimensional spacelike submanifolds through the light cone $\Lambda^{n+1}_{+}(c)$ in $(n+2)$-dimensional De Sitter spacetime $\S^{n+2}_1(c)$ of (constant) sectional curvature $c>0$ (or through the light cone $\Lambda^{n+1}_{+}(-c)$ in $(n+2)$-dimensional anti De Sitter spacetime $\H^{n+2}_1(-c)$ of sectional curvature $-c$) and $n$-dimensional spacelike submanifolds in $\L^{n+2}$ through $\Lambda^{n+1}_{+}$ \cite[Prop. 3.2]{CFP}. This correspondence has an interesting property, namely: If an $n$-dimensional spacelike submanifold in $\Lambda^{n+1}_{+}(c)\subset \S^{n+2}_1(c)$ (or in $\Lambda^{n+1}_{+}(-c)\subset \H^{n+2}_1(-c)$) has parallel mean curvature vector field, then the mean curvature vector field of the corresponding spacelike submanifold in $\Lambda^{n+1}_{+} \subset \L^{n+2}$ has constant squared length \cite[Prop. 4.1]{CFP}, and therefore, it also has a parallel mean curvature vector field. Of course, the compactness of the spacelike submanifold is also preserved in this correspondence. Consequently, from Theorem \ref{1106A} we obtain,
\begin{corollary}
Let $\psi:M^{n}\rightarrow \S^{m+2}_1(c)$, $c>0$ $($resp. $\psi:M^{n}\rightarrow \H^{m+2}_1(-c)$$)$ be a compact spacelike submanifold through the light cone $\Lambda^{n+1}_{+}(c)$ $($resp. $\Lambda^{n+1}_{+}(-c)$$)$ with parallel mean curvature vector field. Then $M^n$ is totally umbilical.    
\end{corollary}
Secondly, we have the following rigidity result,

\begin{corollary}\label{rigidity}
Let $\psi:\S^{n}\rightarrow \L^{m+2}$ be an isometric immersion of the unit sphere $\S^n$ with its canonical metric that satisfies $\psi(\S^n)\subset \Lambda^{n+1}_{+}$. Then, there exists a rigid motion $F$ of $\L^{n+2}$, that preserves $\Lambda^{n+1}_{+}$, such that $\psi =F \circ i$ where $i : \S^n \rightarrow \L^{n+2}$, $i(x)=(1,x)$, $x\in\S^n$, is the natural isometric embedding of $\S^n$ through $\Lambda^{n+1}_{+}$ in $\L^{n+2}$.
\end{corollary}
\begin{proof}
Taking into account Proposition \ref{equivalence2}, we know that the mean curvature vector field of $\psi$ is parallel. We know that $\psi$ must be an embedding and, from Theorem \ref{1106A}, necessarily  $\psi(\S^n)=\S^n(v,1)$ holds, for some $v\in \mathbb{L}^{n+2}$ satisfying $\langle v,v\rangle=-1$ and $v_0<0$. Now observe that $$i(\S^n)=\{y\in \Lambda^{n+1}_{+}\, :\, \langle -e_0,y\rangle = 1\}=\S^n(-e_0,1).$$ Thus, there exists a rigid motion $F$ of $\L^{n+2}$, preserving $\Lambda^{n+1}_{+}$, such that $F(\S^n(-e_0,1))=$ $\S^n(v,1)$.
\end{proof}

Next, as another application of Theorem \ref{1106A}, we will reprove a classical result by Obata. To relate it with Theorem \ref{1106A} it is convenient to recall the spacelike embedding $i_f \colon \S^{n}\to \Lambda^{n+1}_+$, given for each $f\in \mathbb{S}^n$ in (\ref{global_section2}), with induced metric $g_{f}= e^{2f}g_{0}$ on $\S^n$. 

\begin{corollary}\label{Obata} {\rm \cite[Prop. 6.1]{Ob71}} Let $(\S^n,g_{0})$ be an $n(\geq 2)$-dimensional round sphere with radius $1$, and let $g^*$ be another Riemannian metric on $\S^n$ pointwise conformal to $g_{0}$. Then $g^*$ is of constant scalar curvature $n(n-1)$ if and only if it is of constant sectional curvature $1$.
\end{corollary}
\begin{proof}
If the metric $g^*$ satisfies $g^*=e^{2f}\,g_{0}$, for some $f\in C^{\infty}(\mathbb{S}^n)$. That is, we have $g^*=g_f$. Now, the mean curvature vector field of $i_f$ is parallel because the scalar curvature of $g^*$ is assumed to be constant (Proposition \ref{equivalence2}). According to Theorem \ref{1106A}, we have $i_{f}(\S^{n})=\S^{n}(v,r)$, where $v\in \L^{n+2}$ with $\langle v,v\rangle = -1$, $v_0<0$, and  $r>0$. Finally, our assumption on the value of the scalar curvature  $g^{*}$ implies $r=1$.
\end{proof}

\begin{remark}\label{250225A} {\rm 
We would like to add a comment comparing our (extrinsic) proof in Corollary \ref{Obata} with the (intrinsic) original one by Obata, \cite[Prop. 6.1]{Ob71}. First note that the scalar curvature  $S^*(=S_f)$ of the metric $g^*(=g_f)$ is given by (\ref{10025A}). If we perform the change of variable $h=e^{-f}$, this formula is written as follows  
\begin{equation}\label{remark_Obata}
S^{*}=n(n-1)h^2+2(n-1)h\,\Delta^{0}\, h- n(n-1)\| \nabla^{0} h \|_{0}^2, 
\end{equation}
compare with \cite[Eq. (5.4)]{Ob71}. 
Hence, the condition $S^{*}=n(n-1)$ holds if and only if the function $h(>0)$ satisfies the following differential equation
\begin{equation}\label{10225B}
\Delta^{0} h=\frac{n}{2h}\Big ( \, 1- h^2+\| \nabla^{0} h \|_{0}^2\, \Big),  
\end{equation}
\cite[ Eq. (6.1)]{Ob71}. At this point, Obata showed that the conformal metric $g^{*}$ must be Einstein. Finally, he pointed out that being $g^*$ conformally flat and Einstein, the Riemannian manifold $(\S^{n}, g^{*})$, as a direct consequence of the decomposition of the Riemann curvature tensor, is a space of constant sectional curvature \cite[1.118]{Besse}. Therefore, the metric $g^{*}$ has constant sectional curvature equal to $1$, which completes the remark
}
\end{remark}

It should be noted that the solutions of (\ref{10225B}) are not explicitly provided in \cite{Ob71}. Accordingly, the aim of the final part of this paper is to explicitly determine all solutions of the differential equation (E). In particular, this will yield all positive solutions of (\ref{10225B}).

\vspace{1mm}

\begin{corollary}\label{ultimo} For each positive constant $k$, the only solutions of equation {\rm (E)} are the functions $f: \S^n \rightarrow \R$ given by
\[
f(x)=\log\Big(\,\frac{1/\sqrt{k}}{-v_0+\sum_{i=1}^{n+1}v_ix_i}\,\Big), \quad x\in \S^n,
\]
where $v=(v_0,v_1,...,v_{n+1})\in \L^{n+2}$ with $\langle v,v\rangle = -1$, and $v_{0}<0$. 
\end{corollary}
\begin{proof} For each solution $f$ of equation (E), consider the spacelike embedding $i_f : \mathbb{S}^n \rightarrow \Lambda^{n+1}_+\subset \mathbb{L}^{n+2}$ given by (\ref{global_section2}). The mean curvature vector field $\mathbf{H}_f$ of $i_f$ satisfies $\|\mathbf{H}_f\|^2=k$. Hence, $\nabla^{\perp}\mathbf{H}_f=0$ because of Proposition 
 \ref{equivalence2}. Now, Theorem \ref{1106A} is called to assert that $i_f$ is totally umbilical with $i_f(\mathbb{S}^n)=\mathbb{S}^n(v,1/\sqrt{k})$, where $v\in\mathbb{L}^{n+2}$ such that $\langle v,v\rangle=-1$, $v_0<0$. Therefore, the first coordinate component $(i_f)_0$ of $i_f$ satisfies $(i_f)_0(x)=e^{f(x)}=(1/\sqrt{k})/\big(-v_0+\sum_{i=1}^{n+1}v_ix_i\big)$, for each $x\in\mathbb{S}^n$, and that ends the proof.
\end{proof}

\begin{corollary}\label{equation} 
On the round sphere of constant sectional curvature $1$, $(\S^n,g_0)$, the only smooth positive functions $h : \S^n \rightarrow \R$ that satisfy equation {\rm (\ref{10225B})}
are  given by 
$$ 
h(x)=-v_{0}+\sum_{i=1}^{n+1}v_{i}x_{i}, \quad x\in \S^n,
$$
where $v=(v_0,v_1,...,v_{n+1})\in \L^{n+2}$ with $\langle v,v\rangle = -1$, and $v_{0}<0$. 
\end{corollary}
The proof follows from Corollary \ref{ultimo}, taking $k=1$ and $f=\log (1/h)$ for each given solution $h$ of equation (\ref{10225B}).

\vspace{2mm}

\noindent {\bf Availability of data and materials:} No data or materials are associated with this study.

\vspace*{3mm}

\end{document}